\documentclass[11pt]{article}

\usepackage{amsthm,amssymb,amsmath,euscript}

\newtheorem{thm}{Theorem}[section]
\newtheorem{lemma}{Lemma}[section]
\newtheorem{definition}{Definition}[section]

\newtheorem{pro}{Proposition}[section]

\def\eps{\varepsilon}
\def\R{\mathbb{R}}
\def\C{\mathbb{C}}

\def\Om{\Omega}
\def\d{\dfrac{\mbox{d}}{\mbox{dt}}}
\def\cH{\EuScript H}

\def\ol{\overline}

\begin{document}
%\large

\begin{title}{Well-posedness of magnetic Zakharov system on 2D bounded domain}
\end{title}
\author{Oleksiy S. Shcherbina }
\date{}
\maketitle
\centerline{\em Department of Mathematics and Informatics,
Karazin Kharkiv National University,}
\centerline{\em 4 Svobody Sq. 61077 Kharkiv, Ukraine}
%\centerline{E-mail: shcherbina@mail.ru}
\begin{abstract}
This paper investigates initial-boundary value problems for the magnetic Zakharov system, which arises in plasma physics in two-dimensional domains. We prove the global well-posedness of this problem in certain Sobolev-type spaces and the continuous dependence of the solution on the initial data.
\end{abstract}

\section{Introduction}

We consider the dissipative magnetic Zakharov system in a smooth 2D bounded domain
$\Omega\subset\R^2$ of the form
\begin{equation}
\label{Zakh_magnetic}
\left\{\begin{array}{ll}
i E_t + \Delta E - n E + i E\times B +i\gamma_1 E  = g_1(x), & x\in \Omega,
\\
n_{tt}+\gamma_2 n_t -\Delta\left(n+|E|^2\right)=g_2(x), & x\in \Omega,
\\
B_{tt}-\gamma_3\Delta B_t +\Delta^2\left(B+i E\times \ol{E}\right)=g_3(x), & x\in \Omega.
\end{array}\right.
\end{equation}

Here the function $E=E(x,t):\Omega\times\R_+ \to \C\times\C\times\{0\}$ is the slowly varying amplitude of the
high-frequency electric field, and $n = n(x, t) : \Omega\times\R_+ \to  R$ denotes the fluctuation
of the ion-density from its equilibrium value, and  $B=B(x,t):\Omega\times\R_+ \to\{0\}\times\{0\} \times\R$
is the self-generated
magnetic field. $\overline{E}$ denotes the conjugate complex of $E$, and the notation
$\times$ in (\ref{Zakh_magnetic}) means the cross product for $\R^3$ or $\C^3$ valued vectors. 
Parameters $\gamma_i$ are nonnegative and functions $g_1(x):\Omega\to \{0\}\times\{0\} \times\C$, $g_2(x):\Omega\to \R$ and $g_3(x):\Omega\to \R\times\R \times\{0\}$ are given.

We supplement this system with initial data
\begin{equation}
\label{initial_data}
n_t(x,0)=n_1(x),\,n(x,0)=n_0(x),\,B_t(x,0)=B_1(x),\,B(x,0)=B_0(x),\,E(x,0)=E_0(x),
\end{equation}
and with the boundary conditions
\begin{equation}
\label{bound_cond}
\begin{array}{lll}
E|_{\partial\Omega}=0, & n|_{\partial\Omega}=0, &
B|_{\partial\Omega}=\Delta B|_{\partial\Omega}=0.
\end{array}
\end{equation}

%
%Below $A=-\Delta$ with the Dirichlet boundary conditions and $D(A)=H^2\left(\Omega\right)\bigcap
%H^1_0\left(\Omega\right)$.

After omitting the magnetic field $B$ the system~(\ref{Zakh_magnetic}) is reduced to the standard Zakharov system 
\begin{align}
\label{Z-1}
\left\{\begin{array}{lll}
i E_t+\Delta E-n E+i\gamma_1 E=g_1, & x\in\Om, & t>0,
\\
n_{tt}+\gamma_2 n_t-\Delta\left(n+|E|^2\right)=g_2, & x\in\Om, & t>0,
\end{array}\right.
\end{align}
which was proposed by  Zakharov~\cite{Zakharov} for the description of wave phenomena in plasma. 

\par
In the case $\Om=\R^{d}$ ($d\ge 1$), $\gamma_1=\gamma_2=0$ and $g_1(x)=g_2(x)=0$ the Cauchy problem for (\ref{Z-1})
was studied by many authors (see, e.g.,\/ \cite{Bourgain1,Bourgain2,Ginibre}
and the references therein). The local well-posedness was shown in appropriate
Sobolev classes (see \cite{Ginibre} for the most advanced results)
and there are solutions which blow up in finite time for $d =2$
\cite{Glangetas}. In the case $d\le 3$ for some special class  of initial data
the global well-posedness is also known \cite{Bourgain2}.
\par
In the case when $\Omega\subset \R^d$ is a bounded domain($d=1,2$) and $\gamma_i\ge 0$ the Cauchy problem for (\ref{Z-1}) for $d=1$ was studied by~\cite{Flahaut, Goubet, Shcherbina}. 
The well-posedness was proved of~(\ref{Z-1}) for Dirichlet and periodic bounded conditions, 
existence of the compact global attractor $\mathcal{A}$ and (in the case of periodic boundary conditions) analyticity of elements of $\mathcal{A}$.
In 2D case system~(\ref{Z-1}) was  studied in~\cite{ChuShch} for different types of boundary conditions. 
Global well-posedness was proved in some Sobolev type classes and studied properties of the solutions. 
In the dissipative case the existence of a global attractor is established.
\par
The system~(\ref{Zakh_magnetic}) at first was derived by X.~He~\cite{He} for description of the  pondermotive force and magnetic field 
generation effects resulting from the non-linear interaction between plasma-wave and particles. 
In the case $\Om=\R^{d}$ ($d=2, 3$), $\gamma_1=\gamma_2=0$ and $g_1(x,t)=g_2(x,t)=0$ the Cauchy problem for (\ref{Zakh_magnetic})
was studied by\cite{Zhang}. 
It was proved the local existence and uniqueness of the solution for this system.
\par
In this paper we prove that system~(\ref{Zakh_magnetic}) has a unique solution on $[0, T]$, where either $T\sim \|g_1\|_{L_2}^{-1}$ in a conservative case $\gamma_1=0$ or $T=\infty$ %$\max\left \{\|E_0\|^2, \gamma_1^{-2}\sup\limits_{t\in[0,T]}\|g_1(t)\|^2\right \}<$ 
in the dissipative case $\gamma_1>0$.
To prove the existence of solutions we use Galerkin approximations
and the standard compactness method. The proof of their uniqueness
relies on the method which was developed in the theory of shallow
shells for the case of critical nonlinearities
(see \cite{Sedenko_1} and also \cite{Chueshov, ChuShch}). 
Additionally, it was proved some energy relation for this solution and continuous dependence on initial data.

%Let $U$ be a $\C^3$ valued function on $\Omega$. Then it is straightforward to check that
%$U\times \ol{U}$ is  purely imaginary. Therefore $iE\times \ol{E}$ in the third
%equation in (\ref{Zakh_magnetic}) is real-valued $\R^3$ function with
%zero two first components.

%We would like to recall also the well-known Agmon
%inequality for $u \in H^{2}(\Omega)\bigcap H^{1}_0(\Omega)$
%\begin{equation}
%\label{Agmon}
% \| u \|_{L^{\infty}} \leq C\| u \|^{1/2}
% \|\Delta u \|^{1/2}.
%\end{equation}

\section{Statement of the main results}

Below we denote  by $\|\cdot\|$ and $(\cdot,\cdot)$ the norm and the inner product correspondently in the space  $H=L_2(\Omega)$ and consider
 the  minus Laplace operator  with the Dirichlet boundary conditions:
\[
Au=-\Delta u,~~~ u\in D(A)=H^2\left(\Omega\right)\bigcap
H^1_0\left(\Omega\right).
\]
The operator $A$ is  positive and has a discrete spectrum, i.e.,
there exists an orthonormal basis $\{e_k\}$ in $H$ consisting of eigenfunctions
of $A$:
\[
A e_k=\lambda e_k, ~~k=1,2,\ldots,~~0<\lambda_1\le\lambda_2\le\ldots, ~~\lim_{k\to\infty}\lambda_k=+\infty.
\]
We  also use the notation $H_s=D(A^{s/2})$ for $s\ge 0$ and endow this space
with the graph norm    $\|\cdot \|_s=\|A^{s/2}\cdot\|$.
If $s<0$, then $H_s$ denotes the completion of $H_0$ with respect to
$\|\cdot \|_s=\|A^{s/2}\cdot\|$.

%\par
%It is clear from the interpolation theory \cite{Lions} that
%\[
%H_s=\left\{
%  \begin{array}{ll}
%    (H^s\cap H_0^1)(\Om), & \hbox{for}~~\frac12<s<\frac52,~s\neq\frac32 \\
%    H^s(\Om), & \hbox{for}~~-\frac32<s<\frac12,~s\neq-\frac12.
%  \end{array}
%\right.
%\]
We also mention that we deal with both real and complex versions
of the spaces $H_s$. We keep notation $H_s$ for real case
and denote by $\ol{H_s}$ its complexification. For norms and inner products
we use the same notations.
\par
We understand solutions to problem
(\ref{Zakh_magnetic})--(\ref{initial_data}) in the sense of the following
definition.
\begin{definition}\label{de:sol}
{\rm
A triple  $(n;B;E)$ is said to be a 
%\emph{semi-weak} solution to
%problem
%(\ref{Zakh_magnetic})--(\ref{initial_data})
%on an interval $[0,T]$ iff
%\begin{equation}\label{sol-def}
%(n_t; n; B_t; B; E)\in  L_\infty\left(0,T;\, \cH_0\right),~~\cH_0\equiv
%H_{-1}\times H_0\times {H_{-2}}\times {H_0}\times\overline{H}_1,
%\end{equation} and
%relations (\ref{Zakh_magnetic}) and (\ref{initial_data}) are satisfied in the sense
%of distributions. This solution is called 
{\em semi-strong} solution if
\begin{equation}\label{sol-def-ss}
(n_t; n; B_t; B; E)\in  L_\infty\left(0,T;\, \cH\right),~~\cH\equiv
H_{0}\times H_1\times {H_{-1}}\times {H_1}\times\overline{H}_2.
\end{equation}
}
\end{definition}
\par
We also note that by (\ref{Zakh_magnetic}) and (\ref{sol-def-ss}) we have that
\begin{equation}\label{weak-cont}
n_{tt}\in  L_\infty\left([0,T];\,  H_{-1}\right) \quad B_{tt}\in  L_\infty\left([0,T];\,  H_{-3}\right)
\quad\mbox{and}\quad E_{t}\in L_\infty\left([0,T];\, \overline{H}\right).
\end{equation}
Therefore (after changing on a set of zero measure in  $[0,T]$) the vector $(n_t; n; B_t; B; E)$ is a weakly continuous function with values
in $\cH$ (see~\cite[Lemma~8.1]{Lions}).
In particular, this means that initial data (\ref{initial_data}) make sense.
Moreover, (\ref{sol-def-ss}) and (\ref{weak-cont}), Aubin's embedding  theorem (see \cite[Corollary 4]{sim}) imply that
\begin{equation}\label{Z-def2}
(n_t; n; B_t; B; E_t; E)\in  C\left([0,T];\, H_{-\sigma}\times H_{1-\sigma}
  \times H_{-\sigma}\times {H_{-1-\sigma}}\times {H_{1-\sigma}}\times H_{2-\sigma}\right)
\end{equation}
for every $\sigma>0$.
\par
Our first result is the following assertion on the existence of a semi-strong solution of problem (\ref{Zakh_magnetic})--(\ref{initial_data}).
\begin{thm}
\label{t:exist}
Let $(n_1, n_0, B_1, B_0, E_0)\in \cH$
and $g_1 \in \overline{L_4\left (\Omega\right )}$, $g_2\in L_2\left (\Omega\right )$, $g_3 \in L_2\left (\Omega\right )$.
Then  problem (\ref{Zakh_magnetic})--(\ref{initial_data}) has a
semi-strong solution on an interval $[0,T]$ provided that
\begin{itemize}
    \item either $\gamma_1\ge 0$ and $\|E_0\|+T\cdot \|g_1\|<\dfrac{1}{\sqrt{2}}$,
    \item or else  $\gamma_1>0$, and
    $\max\left \{\|E_0\|, \gamma_1^{-1}\sup\limits_{t\in[0,T]}\|g_1(t)\|\right \}<\dfrac{1}{\sqrt{2}}$.
\end{itemize}
%This solution depends continuously
%on initial data, i.e.,
%\begin{eqnarray}
%\label{c-id}
%\lefteqn{\lim_{k\to\infty}\max_{[0,T]}\left(
%\|n^k_t(t)-n_t(t)\|^2+\| n^k(t)-n(t)\|_1^2\right.
%} \nonumber
%\\
% & & \qquad
%\left.+ \|B^k_t(t)-B_t(t)\|_{-1}^2+\| B^k(t)-B(t)\|_1^2
%+\| E^k(t)-E(t)\|_2^2\right)=0,
%\end{eqnarray}
%where $(n;B;E)$ is the solution with initial data $U\equiv (n_1;n_0;B_1;B_0;E_0)$ and
%$(n^k;B^k;E^k)$ is the solution with initial data $U^k\equiv(n^k_1;n^k_0;B^k_1;B^k_0;E^k_0)$
%such that $U^k\to U$ in the space
%$\cH$ as $k\to\infty$.
\end{thm}
Our next result is the following uniqueness theorem for semi-strong solutions of problem (\ref{Zakh_magnetic})--(\ref{initial_data}).
\begin{thm}
\label{t:uniqueness}
Problem (\ref{Zakh_magnetic})--(\ref{initial_data}) on every interval $[0, T]$ has at most one semi-strong solution.
\end{thm}

And finally, we prove some additional continuous conditions for problem (\ref{Zakh_magnetic})--(\ref{initial_data}).
\begin{thm}
\label{t:continuos}
The semi-strong solution of problem (\ref{Zakh_magnetic})--(\ref{initial_data}) possesses the property
\begin{equation}
\label{t_cont}
\left (n_t, n, B_t, B, E_t, E\right )\in C\left ([0, T], H_{0}\times H_1\times {H_{-1}}\times {H_1}\times\overline{H}\times\overline{H}_2\right ).
\end{equation}
Moreover, this solution depends continuously on initial data, i.e.,
\begin{eqnarray}
\label{c-id}
\lefteqn{\lim_{k\to\infty}\max_{[0,T]}\left(
\|n^k_t(t)-n_t(t)\|^2+\| n^k(t)-n(t)\|_1^2+\|B^k_t(t)-B_t(t)\|_{-1}^2\right.
} \nonumber
\\
 & & \qquad
\left.+\| B^k(t)-B(t)\|_1^2+\| E^k_t(t)-E_t(t)\|^2
+\| E^k(t)-E(t)\|_2^2\right)=0,
\end{eqnarray}
where $(n;B;E)$ is the solution with initial data $U\equiv (n_1;n_0;B_1;B_0;E_0)$ and
$(n^k;B^k;E^k)$ is the solution with initial data $U^k\equiv(n^k_1;n^k_0;B^k_1;B^k_0;E^k_0)$
such that $U^k\to U$ in the space
$\cH$ as $k\to\infty$.
\end{thm}

\section{The proof of existence theorem}
\noindent
Before proceeding to the proof, we recall the following classical embedding in two dimensional case
\begin{equation}
\label{embed}
H^{1-2/p}(\Omega)\subset L_{p}(\Omega),
\quad 2<p<\infty.
\end{equation}
By interpolation, (\ref{embed}) implies that for an any function $u\in H^1_0\left(\Omega\right)$ we get
\begin{equation}
\label{L_p}
\|u\|_{L_p}\le C_p(\Omega)\|u\|^{2/p}\|\nabla u\|^{1-2/p},
\quad 2<p<\infty.
\end{equation}
In particular, for $p=4$ we recall the following Sobolev-type estimate
\begin{lemma}\label{le:1s}
For any $u\in H_1$ remains true
\begin{equation}\label{an-est}
\|u\|^4_{L_4(\Om)}\le 2\|u\|^2\cdot\|\nabla u\|^2.
\end{equation}
\end{lemma}
The proof of this lemma can be found in Appendix of~\cite{ChuShch}.
\par
Let us prove now Theorem~\ref{t:exist} by the compactness method.
We consider the Galerkin  approximations  of the magnetic Zakharov problem:
\begin{equation}
\label{exist_1a}
\left\{\begin{array}{l}
i E_t^N -A E^N - P_N(n^N E^N) + iP_N (E^N\times B^N) +i\gamma_1 E^N  =P_N g_1(x),
\\
n^N_{tt}+\gamma_2 n^N_t +A\left(n^N+P_N|E^N|^2\right)=P_N g_2(x),
\\
B^N_{tt}+\gamma_3A B^N_t +A^2\left(B^N+i P_N(E^N\times \ol{E^N})\right)=P_N g_3(x),
\end{array}\right.
\end{equation}
Here and below $P_N$ is the orthoprojector on Span$\{e_k\, :\, k=1,2,\ldots,N\}$, where $\{e_k\}$ are the eigenfunctions of $A$.
To simplify notations we omit below the index $N$ for functions $(n, B, E)$.
\par
We note that  the product $E\times \ol{E}$ is
purely imaginary.
Therefore 
$\Im \left( i(E\times B)\cdot \ol{E}\right )=\Re\left ((\ol{E}\times E)\cdot B\right )=0$.

In order to prove the initial a priori estimate,
let us multiply the first equation from (\ref{exist_1a}) by $2\ol{E}$
and integrate the imaginary part of the result over $\Omega$. Then
we have
\begin{equation}
\label{exist0}
\d\|E\|^2+2\gamma_1 \|E\|^2=2\Im(g_1, E)\le
2\|E\| \|g_1\|.
\end{equation}
If $\gamma_1=0$ then we obtain that $\d\|E\|\le \|g_1\|$. Hence,
\begin{equation}
\label{apr_est1_1}
\|E(t)\|\le \kappa\equiv \|E\|_0+T\|g_1\|.
\end{equation}
In the case $\gamma_1>0$ we can estimate the r.h.s of~(\ref{exist0}) such as 
\[
2\|E\|\|g_1\|\le \gamma_1\|E\|^2+\gamma_1^{-1}\|g_1\|^2.
\]
After integrating (\ref{exist0}) over $[0,t]$ we have the initial a priori estimate
\begin{eqnarray}
\label{apr_est1}
\|E(t)\|^2\le \|E_0\|^2 e^{-\gamma_1 t}+
\gamma_1^{-1}\int_0^t\|g_1(\tau)\|^2 e^{-\gamma_1 (t-\tau)}\mbox{d}\tau \nonumber
\\ 
\le  \max\left \{\|E_0\|^2, \gamma_1^{-2}\sup\limits_{t\in[0,T]}\|g_1(t)\|^2\right \}\equiv\kappa^2.
\end{eqnarray}

%==================================================================
%\centerline{$\ast\ast\ast$}
%==================================================================
Now we can proceed to prove the next a priori estimate.
Let us define now the functional
\begin{equation}
\label{G_1}
G_1(t)=\|\nabla E\|^2+\dfrac12\|n_t\|^2_{-1}+\dfrac12\|n\|^2
+\dfrac12\|B_t\|^2_{-2}+\dfrac12\|B\|^2+\int_{\Omega}n|E|^2
+i\int_{\Omega}(E\times\ol{E})\cdot B.
\end{equation}
Taking into account that
\begin{eqnarray*}
\d \int_{\Omega}n|E|^2=\int_{\Omega} n_t|E|^2+2\Re\int_{\Omega} nE\cdot \ol{E_t},\\
i\d \int_{\Omega}(E\times\ol{E})\cdot B= 
i\int_{\Omega}(E\times\ol{E})\cdot B_t-2\Re\int_{\Omega}i(E\times B)\cdot \ol{E_t},
\end{eqnarray*}
it is straightforward to check that
%
%Let us multiply
%
%
%(i) the first equation from (\ref{exist_1a}) by $-2\ol{E_t}-2\gamma_1 \ol{E}$
%and integrate the real part of the result over~$\Omega$;
%
%(ii) the second equation from (\ref{exist_1a}) by $A^{-1}n_t$
%and integrate the result over $\Omega$;
%
%(iii) the third equation from (\ref{exist_1a}) by $A^{-2}B_t$
%and integrate the result over $\Omega$;
%
%\noindent
%and  take the sum of these relations,
\begin{equation}
\label{sum1}
\d G_1(t)+2\gamma_1\|\nabla E\|^2+\gamma_2\|n_t\|_{-1}^2
+\gamma_3\|B_t\|_{-1}^2=R_0(t), 
\end{equation}
where
\begin{equation}
\label{R_0}
R_0(t)=2\gamma_1\int_{\Omega}(i E\times B-n E)\cdot \ol{E}
-2\Re\int_{\Omega}g_1\cdot\ol{E_t}
-2\gamma_1\Re\int_{\Omega}g_1\cdot\ol{E}+\int_{\Omega}g_2A^{-1}n_t
+\int_{\Omega}g_3\cdot A^{-2}B_t.
\end{equation}
By integration from (\ref{sum1}) we have that
\begin{equation}
\label{sum1_1}
G_1(t)\le G_1(0)+\int_{0}^{t} R_0(\tau) d\tau.
\end{equation}
Now we estimate from below the functional $G_1(t)$. From H\"{o}lder inequality, (\ref{an-est}) and (\ref{apr_est1}) we get
\begin{equation}
\label{exist2}
\left |\int_{\Omega}n|E|^2
+i\int_{\Omega}(E\times\ol{E})\cdot B\right|
\le \left (\|n\|+\|B\|\right )\|E\|^2_{L_4}\le\sqrt{2}\kappa (\|n\|+\|B\|)\|\nabla E\|.
\end{equation}
We note that if $\kappa<\dfrac{1}{\sqrt{2}}$ then there exists $\alpha>0$ such that
\[
\|\nabla E\|^2+\dfrac12\|n\|^2 +\dfrac12\|B\|^2-\sqrt{2}\kappa (\|n\|+\|B\|)\|\nabla E\|\ge \alpha \left (\|\nabla E\|^2+\|n\|^2 +\|B\|^2\right ).
\]
Therefore from (\ref{G_1}) we can conclude that
\begin{equation}
\label{exist5}
G_1(t)\ge
\alpha\left (\|\nabla E\|^2+\|n\|^2+\|B\|^2\right )+\dfrac{1}{2}\|n_t\|^2_{-1}+\dfrac{1}{2}\|B_t\|^2_{-2}\equiv H_0(t).
\end{equation}
Now we need to estimate the r.h.s of (\ref{sum1_1}) by $C_1+C_2\int_0^t H_0(\tau)d\tau$. 
Then from (\ref{sum1_1}), (\ref{exist5}) we get
\[
H_0(t)\le C_1+C_2\int_0^t H_0(\tau)d\tau.
\]
From Gronwall's lemma we obtain our second a priori estimate
\begin{equation}
\label{apr_est2}
\|\nabla E\|^2+\|n_t\|^2_{-1}+\|n\|^2+\|B_t\|^2_{-2}+\|B\|^2\le C, \quad
\forall t\in [0,T],
\end{equation}
where $C$ does not depend on $N$.
%We recall that 
%\[
%\left |\int_{\Omega}n|E|^2+i\int_{\Omega}(E\times\ol{E})\cdot B\right|\le \sqrt{2}\kappa (\|n\|+\|B\|)\|\nabla E\|.
%\]

%First of all from (\ref{an-est}) and (\ref{apr_est1}) it follows that
%\begin{equation}
%\label{exist2}
%\begin{array}{rcl}
%\left|\int_{\Omega}(i E\times B-n E)\cdot \ol{E}\right|
%& \le & (\|n\|+\|B\|)\|E\|_{L_4}^2\le \sqrt{2}\kappa (\|n\|+\|B\|)\|\nabla E\|.
%\end{array}
%\end{equation}
To do this, let us estimate $R_0(t)$. 
From the first equation of (\ref{exist_1a}) it follows that 
$E_t=-i AE- i P_N(n E)- P_N(E\times B)-\gamma_1 E-iP_Ng_1$. 
Substituting it into
$\int_{\Omega}g_1\cdot\ol{E_t}$ we get
\begin{equation}
\label{exist3}
\begin{array}{rcl}
\left|\Re\int\limits_{\Omega}g_1\cdot\ol{E_t}\right|
&
\le
&
\|\nabla g_1\|\|\nabla E\|
+(\|n\|+\|B\|)\|E\|_{L_4}\|g_1\|_{L_4}+\gamma_1\|g_1\|\|E\|
\\
& &
\le C(\|n\|+\|B\|+\|\nabla E\|^{1/2})\|\nabla E\|^{1/2}
\le \eps H_0(t)+C_\eps.
\end{array}
\end{equation}
Taking into account that
\[
\left |2\gamma_1\Re\int_{\Omega}g_1\cdot\ol{E}+\int_{\Omega}g_2A^{-1}n_t +\int_{\Omega}g_3\cdot A^{-2}B_t\right |\le 
C\left (1+\|n_t\|^2_{-1}+\|B_t\|^2_{-2}\right ),
\]
from (\ref{exist2}) and (\ref{exist3}) we can conclude the estimation
\[
R_0\le C_1+C_2 H_0(t),
\]
which is enough for us.

%==================================================================
%\centerline{$\ast\ast\ast$}
%==================================================================

Now we are ready to proceed to prove the main a priori estimate.
Let us multiply
the first equation from (\ref{exist_1a}) by $-2A\ol{E_t}-2\gamma_1 A\ol{E}$
and integrate the real part of the result over~$\Omega$
\begin{eqnarray}
\label{exist6}
\d\left\{\|AE\|^2+2\Re\int_\Omega g_1\cdot A\ol{E}\right\}+2\gamma_1\|AE\|^2
+2\Re\int_\Omega n E\cdot A\ol{E_t}-2\Re\int_\Omega i(E\times B)\cdot A\ol{E_t}
\nonumber \\
%& & \qquad \qquad \qquad \qquad
=2\gamma_1\Re\int_\Omega (i E\times B-n E-g_1)\cdot A\ol{E}.
%-2\gamma_1\Re\int_\Omega g_{1}\cdot A\ol{E}.%-2\gamma_1\Re\int_\Omega n E\cdot A\ol{E}
\end{eqnarray}
Let us multiply the second and the third equation from (\ref{exist_1a}) by $n_t$
and by $A^{-1}B_t$ correspondingly, then let us add and integrate the result over~$\Omega$:
\begin{eqnarray}
\label{exist7}
\lefteqn{
\dfrac12\d\left\{\|\nabla n\|^2+\|n_t\|^2+\|\nabla B\|^2+\|B_t\|^2_{-1}\right\}
+\gamma_2\|n_t\|^2+\gamma_3\|B_t\|^2
+\int_\Omega n_t A |E|^2
} \nonumber \\
& & \qquad
+\int_\Omega i A(E\times \ol{E})\cdot B_t
=\int_\Omega n_t P_N g_2+\int_\Omega P_Ng_3\cdot A^{-1}B_t.
\end{eqnarray}
Now we define the functional $H_1(t)$ such that
\begin{equation}
\label{H_1}
H_1(t)=\|AE\|^2+\dfrac12\|\nabla n\|^2+\dfrac12\|n_t\|^2+\dfrac12\|\nabla B\|^2+\dfrac12\|B_t\|^2_{-1}.
\end{equation}
Taking into account (\ref{H_1}) by adding (\ref{exist6}) and (\ref{exist7}) we get
\begin{eqnarray}
\label{exist8}
\lefteqn{
\d\left\{ H_1(t)+2\Re\int_\Omega g_1\cdot A\ol{E}\right\}
+2\gamma_1\|AE\|^2+\gamma_2\|n_t\|^2+\gamma_3\|B_t\|^2+2\Re\int_\Omega n E\cdot A \ol{E_t}
} \nonumber \\
& & \qquad
-2\Re\int_\Omega i(E\times B)\cdot A\ol{E_t}
+\int_\Omega n_t A |E|^2+\int_\Omega i A (E\times \ol{E})\cdot B_t
=R_1(t),
\end{eqnarray}
where
\begin{eqnarray}
\label{R_1}
%\lefteqn{ 
R_1(t)  =  2\gamma_1\Re\int_\Omega (i E\times B-n E-g_1)\cdot A\ol{E} +\int_\Omega n_t P_N g_2+\int_\Omega P_Ng_3\cdot A^{-1}B_t.
%} 
%\nonumber \\
%%& & \qquad
\end{eqnarray}
It is straightforward to check that
\begin{eqnarray}
\label{exist9}
\lefteqn{
\int_\Omega n_t A |E|^2+2\Re\int_\Omega n E\cdot A \ol{E_t}
-2\Re\int_\Omega i(E\times B)\cdot A\ol{E_t}
+\int_\Omega i A(E\times \ol{E})\cdot B_t
} \nonumber \\
& & \qquad
=\d\left\{2\Re\int_\Omega n E\cdot A\ol{E}+\|P_N(n E)\|^2 
-2\Re\int_\Omega i(E\times B)\cdot A \ol{E}+\|P_N(E\times B)\|^2
\right.
\nonumber \\
& & \qquad \qquad
\left.-2\Re\int_\Omega iP_N(E\times B)\cdot (n\ol{E})\right\}-R_2(t),
\end{eqnarray}
where
\begin{eqnarray}
\label{R_2}
\lefteqn{ R_2(t)  =  2\int_\Omega n_t |\nabla E|^2
+2\int_\Omega i(\partial_1 E\times \partial_1 \ol{E}+\partial_2 E\times \partial_2 \ol{E})\cdot B_t
} \nonumber \\
& & \qquad
+2\Re\int_\Omega P_N(n E-iE\times B)\cdot (n_t\ol{E}+i\ol{E}\times B_t)
\nonumber \\
& & \qquad \qquad
+2\Re\int_\Omega P_N(n \ol{E_t}+i\ol{E_t}\times B)\cdot (i\gamma_1 E-P_Ng_1).
\end{eqnarray}
Let us define now the functional
\begin{eqnarray}
\label{H_2}
%\lefteqn{
H_2(t)=2\Re\int_\Omega (n E+g_1)\cdot A\ol{E}+\|P_N(n E)\|^2 
-2\Re\int_\Omega i(E\times B)\cdot A \ol{E}+\|P_N(E\times B)\|^2
%} 
%\nonumber \\
%%& & \qquad \qquad
%+2\Re\int_\Omega g_1\cdot A\ol{E}.
\end{eqnarray}
From (\ref{exist8}), (\ref{exist9}) and (\ref{H_2}) it follows that
\[
\d\left\{H_1(t)+H_2(t)\right\}
+2\gamma_1\|AE\|^2+\gamma_2\|n_t\|^2+\gamma_3\|B_t\|^2
=R_1(t)+R_2(t).
\]
Integrating this equality over $[0, T]$ yields
\begin{eqnarray}
\label{sum2_1}
H_1(t)+H_2(t)+2\gamma_1\int_0^t\|AE(\tau)\|^2d\tau+\gamma_2\int_0^t\|n_t(\tau)\|^2d\tau+\gamma_3\int_0^t\|B_t(\tau)\|^2d\tau \nonumber
\\
=H_1(0)+H_2(0)+\int_0^t \left (R_1(\tau)+R_2(\tau)\right )d\tau.
\end{eqnarray}
At first we estimate $H_2$ as
\begin{eqnarray}
\label{H_2_est}
\lefteqn{
|H_2(t)|  \le C\|A E\|\|E\|_{L_4}(1+\|n\|_{L_4}+\|B\|_{L_4})
+C\|E\|^2_{L_4}(\|n\|_{L_4}+\|B\|_{L_4})^2
} \nonumber \\
& & \qquad \qquad
\le C\|A E\|(1+\|\nabla n\|^{1/2}+\|\nabla B\|^{1/2})+C\|\nabla n\|+C\|\nabla B\|\nonumber 
\\
& & \qquad \qquad\qquad
\le \delta H_1(t)+C_\delta,
\end{eqnarray}
where $\delta>0$ is arbitrary.
Now we estimate the terms in the r.h.s of (\ref{sum2_1}). We note that (\ref{L_p}) and
(\ref{apr_est2}) imply that $\|E\|_{L_p}\le C$ for all $2\le p<\infty$. Therefore,
choosing sufficiently small $\eps$,
from (\ref{R_1}), (\ref{L_p}) and (\ref{apr_est2}) we get
\begin{equation}
\label{R_1_est}
\begin{array}{rcl}
|R_1(t)| & \le & C\|\Delta E\|\|E\|_{L_4}(\|n\|_{L_4}+\|B\|_{L_4})
+C\|\Delta E\|+C\|n_t\|+C\|B_t\|
\\
& \le &
C\|\Delta E\|(1+\|\nabla n\|^{1/2}+\|\nabla B\|^{1/2})+C\|n_t\|+C\|B_t\|
\\
& \le &
C H_1(t)+\varepsilon\|B_t\|^2+C_\eps.
\end{array}
\end{equation}
%===============
In the same way we obtain
\begin{eqnarray}
\label{R_2_est}
\lefteqn{
|R_2(t)|  \le C\|\nabla E\|^2_{L_4}(\|n_t\|+\|B_t\|)
+C(\|n_t\|+\|B_t\|)\|E\|_{L_4}\|E\|_{L_8}(\|n\|_{L_8}+\|B\|_{L_8})
} \nonumber \\
& & \qquad
+C\|\Delta E- n E + iP_N (E\times B) +i\gamma_1 E-P_N g_1(x,t)\|(\|n\|_{L_4}+\|B\|_{L_4})
\nonumber \\
& & \qquad \qquad
\le C(\|\Delta E\|^2+\|n_t\|^2)+\eps\|B_t\|^2+C(\|\nabla n\|^{3/2}+\|\nabla B\|^{3/2})+C
 \\
& & \qquad \qquad\qquad
\le C H_1(t)+\varepsilon\|B_t\|^2+C.
\nonumber
\end{eqnarray}
%and for every $\delta>0$
%\begin{eqnarray}
%\label{H_2_est}
%\lefteqn{
%|H_2(t)|  \le C\|\Delta E\|\|E\|_{L_4}(1+\|n\|_{L_4}+\|B\|_{L_4})
%+C\|E\|^2_{L_4}(\|n\|_{L_4}+\|B\|_{L_4})^2
%} \nonumber \\
%& & \qquad \qquad
%\le C\|\Delta E\|(1+\|\nabla n\|^{1/2}+\|\nabla B\|^{1/2})+C\|\nabla n\|+C\|\nabla B\|
%\\
%& & \qquad \qquad\qquad
%\le \delta H_1(t)+C_\delta.
%\nonumber
%\end{eqnarray}
From (\ref{sum2_1}), (\ref{R_1_est}), (\ref{H_2_est}) and (\ref{R_2_est}) we get
\begin{eqnarray}
\label{exist10}
(1-\delta)H_1(t)+(\gamma_3-2\varepsilon)\int_0^t\|B_t(\tau)\|^2d\tau
\le C_1+C_2\int_0^t H_1(\tau) d\tau.
\end{eqnarray}
From Gronwall's lemma and this relation we conclude the final a priori estimate
\begin{equation}
\label{apr_est3}
\|AE\|^2+\dfrac12\|\nabla n\|^2+\dfrac12\|n_t\|^2+\dfrac12\|\nabla B\|^2+\dfrac12\|B_t\|_{-1}^2
\le C, \quad \forall t\in [0,T].
\end{equation}
Furthermore, we note that (\ref{exist10}) and (\ref{apr_est3}) imply an additional a priori estimate.
\begin{equation}
\label{B_t}
\int\limits_0^T \|B_t(\tau)\|^2 d\tau\le C.
\end{equation}
Taking into account (\ref{apr_est3}) we can pass to the limit as $N\to\infty$  in (\ref{exist_1a}) and obtain the existence of semi-strong solution on $[0, T]$.

\section{The proof of uniqueness theorem}

We use the method suggested by
Sedenko~\cite{Sedenko_1} in the theory of elastic  shells
(see also \cite{Chueshov, ChuShch}).  Bellow we need the following lemma, where
the orthoprojector $P_N$ is the same as in (\ref{exist_1a}).
\begin{lemma}\label{lemma_ineq}
(a)~Let $u(x)\in H^1_0(\Omega)$. Then there exists $N_0>0$ such that
\begin{equation}\label{p-N}
\max\limits_{x \in \Omega}|(P_N u)(x)|\le
C\left[\log(1+\lambda_N)\right]^{1/2}\|u\|_{H^1}
\end{equation}
for all $N>N_0$, where constant $C$ does not depend on $N$.
\par
(b)~Let $u\in H^{s}(\Omega)$ and
$v\in H^{1}(\Omega)\bigcap L_\infty(\Omega)$,
where $0<s<1$. Then
\begin{equation}\label{prod}
 \|u v\|_{H^s}\le C\|u\|_{H^s}\left(\|v\|_{L_\infty}+\|v\|_{H^1}\right).
\end{equation}
(c)~Let $u\in H^{-s}(\Omega)$ and
$v\in H^{1}(\Omega)\bigcap L_\infty(\Omega)$,
where $0<s<1$. Then
\begin{equation}\label{prod2}
 \|u v\|_{H^{-s}}\le C\|u\|_{H^{-s}}\left(\|v\|_{L_\infty}+\|v\|_{H^1}\right).
\end{equation}
\end{lemma}
\begin{proof}
To obtain (\ref{p-N}) one can use the same argument as in
\cite[Lemma 2.2]{Chueshov},
where this inequality was proved for the
spectral projector $P_N$ related with the biharmonic operator.
\par
As for relation (\ref{prod}) it follows from the standard description
of fractional Sobolev spaces (see, e.g.,\ \cite{Lions}) and from the
embedding
$H^1(\Om)\subset L_p(\Om)$ for all $1\le p<\infty$.
\par
Let us prove (\ref{prod2}). From the definition $\|\cdot\|_{H^{-s}}$ and
(\ref{prod}) we get
\[
\|u v\|_{H^{-s}}=\sup\limits_{\|\varphi\|_{H^s}=1}\left|\int_\Omega uv \varphi\right|
\le \|u\|_{H^{-s}}\|v \varphi\|_{H^s}\le
C\|u\|_{H^{-s}}\left(\|v\|_{L_\infty}+\|v\|_{H^1}\right).
\]
\end{proof}

\par
Let $(E^{(i)}), n^{(i)}; B^{(i)})$, $i=1,2$, be two solutions to
problem (\ref{Zakh_magnetic})--(\ref{bound_cond}). We set $E=E^{(1)}-E^{(2)}$, $n=n^{(1)}-n^{(2)}$ and
$B=B^{(1)}-B^{(2)}$. Then the triple  $(E; n; B)$ solves the system
\begin{equation}
\label{uniq_1}
\left\{\begin{array}{l}
iE_t-AE=G_1(t)\equiv n^{(1)}E+nE^{(2)}-iE^{(1)}\times B-i E\times B^{(2)}+i\gamma_1E,
\\
n_{tt}+A n=G_2(t)\equiv -A\left(E^{(1)}\cdot \ol{E}+E\cdot \ol{E^{(2)}}\right)-\gamma_2n_t,
\\
B_{tt}+\gamma_3A B_t+A^2B=G_3(t)\equiv-iA^2\left(E^{(1)}\times \ol{E}+E\times \ol{E^{(2)}}\right).
\end{array}\right.
\end{equation}
with zero initial data.
\par
Let us consider the solution $(E^N,n^N,B^N)$ of the linear finite-dimensional problem
\begin{equation}
\label{uniq_1_N}
\left\{\begin{array}{l}
iE^N_{tt}-AE^N_t=P_NG_{1t}(t),
\\
iE^N_t-AE^N=P_NG_1(t),
\\
n^N_{tt}+A n^N=P_NG_2(t),
\\
B^N_{tt}+\gamma_3AB^N_t+A^2B^N=P_NG_3(t).
\end{array}\right.
\end{equation}
with zero initial data.
\par
Multiplying the first two equations in (\ref{uniq_1_N}) by $2A^{-s}E^N_t$ and $2A^{1-s}E^N$
in $\ol{H_0(\Omega)}$ correspondingly and taking the imaginary part it can easily be seen that
\begin{eqnarray}
\d \left\{\|E^N_t\|^2_{-s}+\|E^N\|^2_{1-s}\right\}
\le 2\left(\|P_NG_{1,t}\|_{-s}\|E^N_t\|_{-s}+\|P_NG_{1}\|_{1-s}\|E^N\|_{1-s}\right).
\nonumber
\end{eqnarray}
Since $P_N\to I$ strongly in every space $H_\sigma$, this implies that
\begin{eqnarray}
\label{uniq_2_1}
\|E_t(t)\|^2_{-s}+\|E(t)\|^2_{1-s}
\le 2\int_0^t\left(\|G_{1,t}(\tau)\|_{-s}\|E_t(\tau)\|_{-s}+\|G_{1}(\tau)\|_{1-s}\|E(\tau)\|_{1-s}\right) d\tau.
\end{eqnarray}
\par
Similarly, multiplying the last two equations in (\ref{uniq_1_N}) by
$2A^{-s}n^N_t$ and $2A^{-1-s}B^N_t$ in $H_0(\Omega)$  correspondingly, we obtain
\begin{eqnarray}
%\label{uniq_2}
\lefteqn{
\d \left\{\|n^N_t\|^2_{-s}+\|n^N\|^2_{1-s}
+\|B^N_t\|^2_{-1-s}+\|B^N\|^2_{1-s}\right\}+2\gamma_3\|B^N_t\|^2_{-s}
} \nonumber \\
& & \qquad
\le 2\left(\|P_NG_{2}\|_{-s}\|n^N_t\|_{-s}+\|P_NG_{3}\|_{-2-s}\|B^N_t\|_{-s}
\right).
\nonumber
\end{eqnarray}
After integration with respect to $t$ and limiting transition $N\to \infty$ we obtain that
\begin{eqnarray}
\label{uniq_2_2}
\lefteqn{
\|n_t(t)\|^2_{-s}+\|n(t)\|^2_{1-s}
+\|B_t(t)\|^2_{-1-s}+\|B(t)\|^2_{1-s}+2\gamma_3\int_0^t\|B_t(\tau)\|^2_{-s} d\tau
} \nonumber \\
& & \qquad
\le 2\int_0^t\left(\|G_{2}(\tau)\|_{-s}\|n_t(\tau)\|_{-s}+\|G_{3}(\tau)\|_{-2-s}\|B_t(\tau)\|_{-s}\right) d\tau.
\end{eqnarray}

\subsection*{Estimate for $\| G_{1,t}\|_{-s}$.}
From definition of $G_1(t)$ we get
\begin{eqnarray}
\label{G_1t}
\lefteqn{
\|G_{1t}\|_{-s} \le \|n^{(1)}_tE\|_{-s}+\|nE^{(2)}_t\|_{-s}+\|E^{(1)}_t\times B\|_{-s}
+ \|E\times B^{(2)}_t\|_{-s}
} \nonumber \\
& & \qquad
+\|n_tE^{(2)}\|_{-s}+\|E^{(1)}\times B_t\|_{-s}+\|n^{(1)}E_t\|_{-s}+\|E_t\times B^{(2)}\|_{-s}
+\gamma_1\|E_t\|_{-s}.
\end{eqnarray}
First, we recall that, for every $0<s<1$, the following embeddings hold on a bounded two-dimensional domain $\Omega$:
\begin{equation}
\label{embedd_s}
L_{2/(1+s)}(\Omega)\subset H^{-s}(\Omega) \mbox{ and } H^{1-s}(\Omega)\subset L_{2/s}(\Omega).
\end{equation}
Since H?lder's inequality with $p=s+1$ and $q=(s+1)/s$ implies that
\begin{equation}
\label{est-s}
\|u\cdot v\|_{-s}\le C\|u\cdot v\|_{L_{2/(s+1)}}\le C\|u\| \|v\|_{L_{2/s}}
\le C\|u\|\|v\|_{1-s}.
\end{equation}
Therefore, (\ref{est-s}) allows to estimate the first four terms in the r.h.s of (\ref{G_1t}) as
\begin{eqnarray}
\label{G_1t_est1}
\lefteqn{
\|n^{(1)}_tE\|_{-s}+\|nE^{(2)}_t\|_{-s}+\|E^{(1)}_t\times B\|_{-s}
+ \|E\times B^{(2)}_t\|_{-s}
} \nonumber \\
& & \qquad
\le C \left(\|n^{(1)}_t\| \|E\|_{1-s}+\|E^{(2)}_t\|\|n\|_{1-s}+\|E^{(1)}_t\|\|B\|_{1-s}
+ \|B^{(2)}_t\|\|E\|_{1-s}\right).
\end{eqnarray}
At second, using (\ref{apr_est3}) and (\ref{prod2}) we estimate the next two term
in the r.h.s of (\ref{G_1t}) as
\begin{equation}
\label{G_1t_est2}
\|n_tE^{(2)}\|_{-s}+\|E^{(1)}\times B_t\|_{-s}
\le C(\|n_t\|_{-s}+\|B_t\|_{-s}).
\end{equation}
Let us recall now that for $Q_N=I-P_N$ and for $0< s_1< s_2$ the following relation remains true
\begin{equation}
\label{Q_N}
\|Q_N u\|_{s_1}\le C\lambda_{N+1}^{-(s_2-s_1)/2}\|Q_N u\|_{s_2},
\quad \forall u\in H^{s_2}(\Omega).
\end{equation}
Then, taking into account (\ref{apr_est3}), (\ref{prod2}), (\ref{est-s}) and (\ref{Q_N}), we estimate the
last two nonlinear term in the  r.h.s of (\ref{G_1t}) as
\begin{eqnarray}
\label{G_1t_est3}
\lefteqn{
\|n^{(1)}E_t\|_{-s}+\|E_t\times B^{(2)}\|_{-s}
} \nonumber \\
& & \quad
\le \|P_N(n^{(1)})E_t\|_{-s}+\|P_N(B^{(2)})E_t\|_{-s}
+\|Q_N(n^{(1)})E_t\|_{-s}+\|Q_N(B^{(2)})E_t\|_{-s}
\nonumber \\
& & \quad\quad
\le C\|E_t\|_{-s}\left(1+\left[\log(1+\lambda_N)\right]^{1/2}\right)
+C\|E_t\|\left(\|Q_Nn^{(1)}\|_{1-s}+\|Q_NB^{(2)}\|_{1-s}\right)
 \\
& & \quad\quad\quad
\le
C\|E_t\|_{-s}\left(1+\left[\log(1+\lambda_N)\right]^{1/2}\right)
+C\lambda_{N+1}^{-s/2}.
\nonumber
\end{eqnarray}
\par
Thus from (\ref{G_1t}), (\ref{G_1t_est1}), (\ref{G_1t_est2}) and (\ref{G_1t_est3})
we obtain the finally estimate for $\|P_NG_{1,t}\|_{-s}$
\begin{eqnarray}
\label{G_1t_est}
\lefteqn{
\|P_NG_{1,t}\|_{-s}\le
C\|E_t\|_{-s}\left(1+\left[\log(1+\lambda_N)\right]^{1/2}\right)
+C(1+\|B_t^{(2)}\|)\|E\|_{1-s}
}\nonumber
\\
& & \qquad
+C\|B_t\|_{-s}+C(\|n\|_{1-s}+\|B\|_{1-s}+\|n_t\|_{-s})
+C\lambda_{N+1}^{-s/2}.
\end{eqnarray}

\subsection*{Estimate for $\|G_{1}\|_{1-s}$.}
From definition of $G_1(t)$ it follows that
\begin{eqnarray}
\label{P_NG_1}
\lefteqn{
\|G_{1}\|_{1-s} \le \|n^{(1)}E\|_{1-s}+\|nE^{(2)}\|_{1-s}+\|E^{(1)}\times B\|_{1-s}
}\nonumber
\\
& & \qquad\qquad
+ \|E\times B^{(2)}\|_{1-s}
+\gamma_1\|E\|_{1-s}.
\end{eqnarray}
It is obvious that
\begin{eqnarray}
\label{P_NG_1-est1}
\lefteqn{
\|n^{(1)}E\|_{1-s}+\|E\times B^{(2)}\|_{1-s}
\le
\|P_N(n^{(1)})E\|_{1-s}+\|E\times P_N(B^{(2)})\|_{1-s}
}\nonumber
\\
& & \qquad\qquad
+\|Q_N(n^{(1)})E\|_{1-s}+\|E\times Q_N(B^{(2)})\|_{1-s}.
\end{eqnarray}
By Lemma~(\ref{lemma_ineq}) we get
\begin{eqnarray}
\label{P_NG_1-est2}
\lefteqn{
\|P_N(n^{(1)})E\|_{1-s}+\|E\times P_N(B^{(2)})\|_{1-s}
\le
\|E\|_{1-s}\left(\max\limits_{x\in\Omega}|P_N(n^{(1)})|+\|P_N(n^{(1)})\|_{H^1}\right.
}\nonumber
\\
& & \qquad
\left.+\max\limits_{x\in\Omega}|P_N(B^{(2)})|+\|P_N(B^{(2)})\|_{H^1}\right)
\le C\|E\|_{1-s}\left(1+\left[\log(1+\lambda_N)\right]^{1/2}\right).
\end{eqnarray}
From Lemma~(\ref{lemma_ineq}) and (\ref{Q_N}) we obtain
\begin{eqnarray}
\label{P_NG_1-est3}
\lefteqn{
\|Q_N(n^{(1)})E\|_{1-s}+\|E\times Q_N(B^{(2)})\|_{1-s}
}\nonumber
\\
& & \qquad
\le
\left(\|Q_N(n^{(1)})\|_{1-s}+\|Q_N(B^{(2)})\|_{1-s}\right)\left(\max\limits_{x\in\Omega}|E|
+\|E\|_{H^1}\right)
\le C\lambda_{N+1}^{-s/2}.
\end{eqnarray}
Therefore, (\ref{P_NG_1-est1})-(\ref{P_NG_1-est3}) imply that
\begin{equation}
\label{P_NG_1-est4}
\|n^{(1)}E\|_{1-s}+\|E\times B^{(2)}\|_{1-s}
\le
C\|E\|_{1-s}\left(1+\left[\log(1+\lambda_N)\right]^{1/2}\right)
+C\lambda_{N+1}^{-s/2}.
\end{equation}
Then, using Lemma~(\ref{lemma_ineq}) we estimate the next two terms from  the right hand
side of (\ref{P_NG_1}) such as
\begin{eqnarray}
\label{P_NG_1-est5}
\lefteqn{
\|nE^{(2)}\|_{1-s}+\|E^{(1)}\times B\|_{1-s}
}\nonumber
\\
& & \qquad
\le
\|n\|_{1-s}\left(\max\limits_{x\in\Omega}|E^{(2)}|+\|E^{(2)}\|_{H^1}\right)
+\|B\|_{1-s}\left(\max\limits_{x\in\Omega}|E^{(1)}|+\|E^{(1)}\|_{H^1}\right)
\\
& & \qquad\qquad
\le C\left(\|n\|_{1-s}+\|B\|_{1-s}\right).
\nonumber
\end{eqnarray}
Thus, (\ref{P_NG_1-est4}) and (\ref{P_NG_1-est5}) imply the finally estimate for $\|G_1\|_{1-s}$
\begin{eqnarray}
\label{P_NG_1-est}
\lefteqn{
\|G_{1}\|_{1-s} \le
C\|E\|_{1-s}\left(1+\left[\log(1+\lambda_N)\right]^{1/2}\right)
}\nonumber
\\
& & \qquad\qquad
+C\left(\|n\|_{1-s}+\|B\|_{1-s}\right)+C\lambda_{N+1}^{-s/2}.
\end{eqnarray}

\subsection*{Estimate for $\|G_{2}\|_{-s}$.}
Definition of $G_2(t)$ implies that
\begin{eqnarray}
\label{P_NG_2}
\|G_{2}\|_{-s} \le
\|E^{(1)}\cdot \ol{E}\|_{2-s}+\|E\cdot \ol{E^{(2)}}\|_{2-s}+\gamma_2\|n_t\|_{-s}.
\end{eqnarray}
Using Lemma~(\ref{lemma_ineq}) and taking into account that the first equation of
(\ref{uniq_1}) and (\ref{est-s}) imply
\[
\|E\|_{2-s}\le \|E_t\|_{-s}+\|G_1(t)\|_{-s}
\le C(\|E_t\|_{-s}+\|E\|_{1-s}+\|n\|_{1-s}+\|B\|_{1-s}),
\]
we deduce
\begin{eqnarray}
\label{P_NG_2-est1}
\lefteqn{
\|E^{(1)}\cdot \ol{E}\|_{2-s} \le
\|\nabla E^{(1)}\cdot \ol{E}\|_{1-s}+\|E^{(1)}\cdot \nabla\ol{E}\|_{1-s}
}\nonumber
\\
& & \qquad
\le C\|E\|_{1-s}\left(\max\limits_{x\in\Omega}|P_N(\nabla E^{(1)})|+\|P_N(\nabla E^{(1)})\|_{H^1}\right)
\nonumber
\\
& & \qquad
+C\|Q_N(\nabla E^{(1)})\|_{1-s}\left(\max\limits_{x\in\Omega}|E|+\|E\|_{H^1}\right)
+C\|E\|_{2-s}\left(\max\limits_{x\in\Omega}|E^{(2)}|+\|E^{(2)}\|_{H^1}\right)
\nonumber
\\
& & \qquad\qquad
\le C(\|E_t\|_{-s}+\|n\|_{1-s}+\|B\|_{1-s})
+C\|E\|_{1-s}\left(1+\left[\log(1+\lambda_N)\right]^{1/2}\right)
+C\lambda_{N+1}^{-s/2}.\nonumber
\end{eqnarray}
Therefore the final estimate for $\|P_NG_{2}\|_{-s}$ is the following
\begin{eqnarray}
\label{P_NG_2-est}
\lefteqn{
\|G_{2}\|_{-s} \le
C(\|E_t\|_{-s}+\|n\|_{1-s}+\|B\|_{1-s})
}\nonumber
\\
& & \qquad\qquad
+C\|E\|_{1-s}\left(1+\left[\log(1+\lambda_N)\right]^{1/2}\right)
+C\lambda_{N+1}^{-s/2}.
\end{eqnarray}
\subsection*{Estimate for $\| G_{3}\|_{-2-s}$.}
Definition of $G_2(t)$ implies that
\begin{eqnarray*}
%\label{P_NG_3}
\|G_{2}\|_{-2-s} \le
\|E^{(1)}\times \ol{E}\|_{2-s}+\|E\times \ol{E^{(2)}}\|_{2-s}.
\end{eqnarray*}
Using the same argument, as for $\|E^{(1)}\cdot \ol{E}\|_{2-s}$ we obtain
\begin{eqnarray}
\label{P_NG_3-est}
\lefteqn{
\|G_{3}\|_{-2-s} \le
C(\|E_t\|_{-s}+\|n\|_{1-s}+\|B\|_{1-s})
}\nonumber
\\
& & \qquad\qquad
+C\|E\|_{1-s}\left(1+\left[\log(1+\lambda_N)\right]^{1/2}\right)
+C\lambda_{N+1}^{-s/2}.
\end{eqnarray}

\subsection*{Concluding step.}
By (\ref{uniq_2_1}), (\ref{uniq_2_2}), (\ref{G_1t_est}), (\ref{P_NG_1-est}),
(\ref{P_NG_2-est})   and (\ref{P_NG_3-est}) we have that
\begin{eqnarray}
\label{uniq_3}
\lefteqn{
\|n_t(t)\|^2_{-s}+\|n(t)\|^2_{1-s}+\|E_t(t)\|_{-s}^2+\|E(t)\|_{1-s}^2
+\|B_t(t)\|^2_{-1-s}+\|B(t)\|^2_{1-s}+2\gamma_3\int_0^t\|B_t(\tau)\|^2_{-s} d\tau
}\nonumber
\\
& & \qquad\qquad
\le
C\left(1+\left[\log(1+\lambda_N)\right]^{1/2}\right)\int_0^t \left(\|E_t(\tau)\|_{-s}^2
+\|E(\tau)\|_{1-s}^2+\|n_t(\tau)\|_{-s}^2\right) d\tau
\nonumber
\\
& & \qquad\qquad
+C\left(1+\left[\log(1+\lambda_N)\right]^{1/2}\right)\int_0^t
\|E(\tau)\|_{1-s}\|B_t(\tau)\|_{-s} d\tau
\nonumber
\\
& & \qquad\qquad
+C\int_0^t\left(\|n(\tau)\|_{1-s}^2+\|B(\tau)\|_{1-s}^2\right) d\tau
+C\lambda_{N+1}^{-s/2}.
\end{eqnarray}
Taking into account that
\begin{eqnarray}
%\label{uniq_3}
\lefteqn{
C\left(1+\left[\log(1+\lambda_N)\right]^{1/2}\right)\int_0^t
\|E(\tau)\|_{1-s}\|B_t(\tau)\|_{-s} d\tau}
\nonumber
\\
& & \qquad\qquad
\le 2\gamma_3\int_0^t\|B(\tau)\|_{1-s}^2 d\tau
+C(1+\log(1+\lambda_N))\int_0^t\|E(\tau)\|_{1-s}^2 d\tau,
\nonumber
\end{eqnarray}
(\ref{uniq_3}) implies that
\begin{eqnarray}
\label{uniq_4}
\psi(t)\le C_1\log(1+\lambda_N)\int_0^t\psi(\tau) d\tau+
C_2\lambda_{N+1}^{-s/2},
\end{eqnarray}
where
\begin{eqnarray}
\label{uniq_psi}
\psi(t)\equiv \|n_t(t)\|^2_{-s}+\|n(t)\|^2_{1-s}+\|E_t(t)\|_{-s}^2+\|E(t)\|_{1-s}^2
+\|B(t)\|^2_{1-s}.
\end{eqnarray}
Using Gronwall's lemma we infer that
\begin{eqnarray}
\label{uniq_5}
\psi(t)\le
C_2\lambda_{N+1}^{-s/2}(1+\lambda_N)^{C_1T_0}
\end{eqnarray}
for every $t\in [0,T_0]$, where $T_0<s/(2C_1)$. Now we send $N\to\infty$ and obtain that
$\psi(t)\equiv 0$ for every $t\in [0,T_0]$. Taking into account (\ref{uniq_psi}) we conclude that
$(E^{(1)}), n^{(1)}; B^{(1)})\equiv (E^{(2)}), n^{(2)}; B^{(2)})$ on $[0,T_0]$. Then we consider
the system (\ref{uniq_1_N}) with zero initial data at the moment $t=T_0$. Arguing as above, we get
$\psi(t)\equiv 0$ for every $t\in [T_0,2T_0]$. Thus, step by
step we get $(E^{(1)}, n^{(1)}; B^{(1)})\equiv (E^{(2)}, n^{(2)}; B^{(2)})$ on $[0,T]$.

\section{Smoothness}
Let us prove now (\ref{c-id}). Let $(n(t); B(t); E(t))$ be  a semi-strong solution to
problem   (\ref{Zakh_magnetic})--(\ref{bound_cond}) on an interval $[0,T]$.
In this case the function $n(t)$ is the weak (variational) solution to
the linear problem
\begin{equation}\label{n-eq}
n_{tt}-\Delta n=F_1(t)\equiv \Delta |E(t)|^2-\gamma_2 n_t(t)+g_2(x,t).
\end{equation}
It is clear to see that $F_1(t)\in L_\infty\left(0,T;L_2\left(\Omega\right)\right)$. Therefore, from \cite[Theorem~8.2]{Lions} it follows that 
\begin{equation}\label{n-sm}
(n_t; n)\in  C\left([0,T];\, L_2\left(\Omega\right)\times H^{1}_0\left(\Omega\right)\right).
\end{equation}
Similarly, the function $B$ is the weak solution to the linear problem 
\begin{equation}\label{B-eq}
B_{tt}+\Delta^2 B=F_2(t)\equiv -\Delta^2\left(i E(t)\times \overline{E}\right)+\gamma_3 \Delta B_t(t)+g_3(x,t).
\end{equation}
Taking into account that $F_2(t)\in L_2\left(0,T;H_{-3}\left(\Omega\right)\right)$ we can also use \cite[Theorem~8.2]{Lions} and conclude that 
\begin{equation}\label{B-sm}
(B_t; B)\in  C\left([0,T];\, H_{-1}\left(\Omega\right)\times H_{1}\left(\Omega\right)\right).
\end{equation}
Let $E_1=-i\Delta E_0+in_0E_0-E_0\times B_0+ig_1$ and $\tilde{E}=E_t$. 
Then $\tilde{E}$ is a solution (in the sense of distributions) of the following problem
\begin{equation}\label{E_t-eq}
i\tilde{E}_t+\Delta\tilde{E}-n\tilde{E}+i\tilde{E}\times B=F_2(t)\equiv n_t E-i E\times B_t+g_{1,t}
\quad \tilde{E}(0)=E_1.
\end{equation}
Note that (\ref{B_t}), (\ref{n-sm}) and (\ref{B-sm}) imply that $F_2(t)\in L_2\left(0,T;L_2\left(\Omega\right)\right)$ and $E_1\in L_2\left (\Omega\right )$.

Let us define now the following sesquilinear form
\[
a(t,u,v)=\int\limits_\Omega \nabla u(x)\cdot \nabla \overline{v(x)}+\int\limits_\Omega n(x,t) u(x)\cdot \overline{v(x)}
-i\int\limits_\Omega B(x,t)\cdot u(x)\times \overline{v(x)}
\]
on $\overline{H_0^1\left (\Omega\right )}$.
Taking into account that 
\begin{eqnarray*}
\lefteqn{ 
\left |\int\limits_\Omega n(x,t) u(x)\cdot \overline{v(x)}\right |+\left |\int\limits_\Omega B(x,t)\cdot u(x)\times \overline{v(x)}\right |
\le \left (\|n\|+\|B\|\right)\|u\|_{L_4}\|v\|_{L_4}
}
\\
& & \qquad
\le \sqrt{2}\left (\|n\|+\|B\|\right)\|u\|^\frac{1}{2}\|v\|^\frac{1}{2}\|u\|_1^\frac{1}{2}\|v\|_1^\frac{1}{2}
\le \delta \|u\|_1\|v\|_1+\dfrac{2}{\delta}\left (\|n\|+\|B\|\right)^2\|u\|\|v\|
\end{eqnarray*}
for every $\delta>0$, then there exist positive $\omega_1$ and $\omega_2$ such that
\[
a(t,u,u)+\omega_1\|u\|^2\ge \omega_2\|u\|^2_1,
\mbox{ for any } u\in \overline{H_0^1\left (\Omega\right )}.
\]
Therefore from \cite[Theorem 10.1]{Lions} it follows that (\ref{E_t-eq}) is uniquely solvable in distributions and $\tilde{E}\in C\left ([0, T], L_2\left (\Omega\right )\right )$. 
From the first equation of (\ref{Zakh_magnetic}), (\ref{n-sm}) and (\ref{B-sm}) we conclude that $E\in C\left ([0, T], H_2\left (\Omega\right )\right )$. 

Now we can start to prove (\ref{c-id}). We will use the standard approach based on an energy relation (see, e.g., \cite[Chap. 3]{Lions}).
At first we prove the following
\begin{pro}
\label{p:energy-rel}
Let $(n_1, n_0, B_1, B_0, E_0)\in \cH$. Then any semi-strong solution $(n, B, E)$ on an interval $[0, T]$ to problem  (\ref{Zakh_magnetic})-(\ref{initial_data}) possesses the energy relation
\begin{eqnarray}
\label{en-rel}
\lefteqn{V(n_t, n, B_t, B, E_t)+2\gamma_1\int\limits_{0}^{t}\|E_t(\tau)\|^2d\tau+\gamma_2\int\limits_{0}^{t}\|n_t(\tau)\|^2d\tau
+\gamma_3\int\limits_{0}^{t}\|B_t(\tau)\|^2d\tau}
\\ \nonumber
& & \qquad\qquad\qquad\qquad\qquad\qquad\qquad\qquad
=V(n_1, n_0, B_1, B_0, E_1)+\int\limits_{0}^{t}\tilde{R}(\tau) d\tau,
\end{eqnarray}
where 
\begin{eqnarray}
\label{en-rel-v}
V(n_t, n, B_t, B, E_t)=\dfrac{1}{2}\left (\|n_t(t)\|^2+\|\nabla n(t)\|^2+\|B_t(t)\|^2_{-1}+\|B(t)\|^2\right )+\|E_t(t)\|^2
\end{eqnarray}
and
\begin{eqnarray}
\label{en-rel2}
\lefteqn{
\tilde{R}=2\int\limits_{\Omega}n_t\left (n|E|^2+|\nabla E|^2\right )
-2\Im\int\limits_{\Omega}\left (n_tB+nB_t\right )\cdot \left (E\times \overline{E}\right )+2\Re\int\limits_{\Omega} B_t\cdot \left (\left (E\times B\right )\times \overline{E}\right )
}
\\ \nonumber
& & \qquad\qquad\quad
+2i\int\limits_{\Omega}B_t\cdot \left (\partial_1 E\times \partial_1 \overline{E}+\partial_2 E\times \partial_2 \overline{E}\right )
-2\Im\int\limits_{\Omega}B_t\cdot \left (g_1\times \overline{E}\right )+2\Re\int\limits_{\Omega}n_tg_1\cdot \overline{E}
\end{eqnarray}
and $E_1=-i\Delta E_0+in_0E_0-E_0\times B_0+ig_1$.
\end{pro}

\textit{Proof.}
Multiplying equation (\ref{E_t-eq}) in $\overline{L_2\left (\Omega\right )}$ by $\tilde{E}=E_t$ and taking the imaginary part of result we get
\[
\d\|E_t(t)\|^2+2\gamma_1\|E_t(t)\|^2=2\Im\int\limits_\Omega (n_tE-iE\times B_t)\cdot \overline{E_t}.
\]
After integrating over $[0, t]$ we obtain
\begin{equation}
\label{en-rel3}
\|E_t(t)\|^2+2\gamma_1\int\limits_0^t\|E_t(\tau)\|^2 d\tau =\|E_1\|^2+2\Re\int\limits_0^t d\tau\int\limits_\Omega (n_t\overline{E}+i\overline{E}\times B_t)\cdot iE_t 
\end{equation}
for any $t\in [0, T]$.

Multiplying the second equation of (\ref{Zakh_magnetic}) in $L_2\left (\Omega\right )$ by $n_t$ we get
\[
\dfrac{1}{2}\d\left (\|n_t(t)\|^2+\|\nabla n(t)\|^2\right )+\gamma_2\|n_t(t)\|^2=\int\limits_\Omega n_t\Delta |E|^2+\int\limits_\Omega n_t g_2.
\]
After integrating over $[0, t]$ we obtain
\begin{eqnarray}
\label{en-rel4}
\lefteqn{\dfrac{1}{2}\left (\|n_t(t)\|^2+\|\nabla n(t)\|^2\right )+\gamma_2\int\limits_0^t\|n_t(\tau)\|^2d\tau}
\nonumber\\
& & \qquad
=\dfrac{1}{2}\left (\|n_1\|^2+\|\nabla n_0\|^2\right )+\int\limits_0^td\tau\int\limits_\Omega n_t\left (\Delta |E|^2+ g_2\right )
\end{eqnarray}
for any $t\in [0, T]$.
Analogously, from the third equation of (\ref{Zakh_magnetic}) we have for any $t\in [0, T]$
\begin{eqnarray}
\label{en-rel5}
\lefteqn{\dfrac{1}{2}\left (\|B_t(t)\|_{-1}^2+\|\nabla B(t)\|^2\right )+\gamma_3\int\limits_0^t\|B_t(\tau)\|^2d\tau}
\nonumber\\
& & \qquad
=\dfrac{1}{2}\left (\|B_1\|_{-1}^2+\|\nabla B_0\|^2\right )+\int\limits_0^t d\tau\int\limits_\Omega B_t\cdot \left (g_3+i\Delta (E\times\overline{E})\right ).
\end{eqnarray}
Summing (\ref{en-rel3})-(\ref{en-rel5}) we have
\begin{eqnarray}
\label{en-rel6}
\lefteqn{V(n_t, n, B_t, B, E_t)+2\gamma_1\int\limits_{0}^{t}\|E_t(\tau)\|^2d\tau+\gamma_2\int\limits_{0}^{t}\|n_t(\tau)\|^2d\tau
+\gamma_3\int\limits_{0}^{t}\|B_t(\tau)\|^2d\tau}
\\ \nonumber
& & %\quad%\qquad\qquad\qquad\qquad\qquad\qquad
=V(n_1, n_0, B_1, B_0, E_1)+2\Re\int\limits_{0}^{t} d\tau \int\limits_\Omega n_t\overline{E}\cdot (iE_t+\Delta E)
+\int\limits_{0}^{t} d\tau \int\limits_\Omega n_t(|\nabla E|^2+g_2)
\\
\nonumber
& & %\quad 
-2\Im\int\limits_{0}^{t} d\tau \int\limits_\Omega (iE_t+\Delta E)\times B_t\cdot\overline{E}
+\int\limits_{0}^{t} d\tau \int\limits_\Omega B_t\cdot\left (
2i\partial_1 E\times \partial_1 \overline{E}+2i\partial_2 E\times \partial_2 \overline{E}+g_3\right ).
\end{eqnarray}
Since $iE_t+\Delta E=nE-iE\times B-i\gamma_1 E+g_1$ it is straightforward to obtain (\ref{en-rel}) from (\ref{en-rel6}). 

\qed

We use (\ref{en-rel}) for proving the following 
\begin{pro}
\label{p:cont}
Let $(n^k_1, n^k_0, B^k_1, B^k_0, E^k_0)\to (n_1, n_0, B_1, B_0, E_0)\in \cH$ and let $(n^k, B^k, E^k)$ be the semi-strong solution   to problem  (\ref{Zakh_magnetic}) with initial data $(n^k_1, n^k_0, B^k_1, B^k_0, E^k_0)$ on an interval $[0, T]$. 
Then 
\begin{eqnarray}
\label{en-rel7}
\lim\limits_{k\to\infty}\max\limits_{[0,T]} \left |V(n_t, n, B_t, B, E_t)-V(n^k_t, n^k, B^k_t, B^k, E^k_t)\right |=0,
\end{eqnarray}
where $V(n_t, n, B_t, B, E_t)$ is given by (\ref{en-rel-v}).
\end{pro}

\textit{Proof.}
Arguing as in the existence theorem, we can conclude that there exists $C$ such that  for all  $k\in\mathbb{N}$
\begin{equation}
\label{en-rel8}
\int\limits_0^T\|B_t(\tau)\|^2 d\tau+
\sup\limits_{[0,T]} \left ( \|n_t(t)\|^2+\|\nabla n(t)\|^2+\|B_t(t)\|^2_{-1}+\|B(t)\|^2+\|E_t(t)\|^2\right )\le C.
\end{equation}
Therefore using the uniqueness theorem we get that
\[
(n^k(t), n^k(t), B^k(t), B^k(t), E^k(t))\to (n(t), n(t), B(t), B(t), E(t)) \mbox{ weakly}^* \mbox{ in } L_{\infty}(0, T, \cH).
\]
By Aubin's embedding  theorem (see \cite[Corollary 4]{sim})
we also have that
\begin{equation}\label{en-rel9}
\lim_{k\to\infty}\max_{[0,T]}\left(
\|E^k_t(t)-E_t(t)\|_{-\sigma}^2 + \| E^k(t)-E(t)\|_{2-\sigma}^2\right)=0,
\end{equation}
and
\begin{equation}\label{en-rel10}
\lim_{k\to\infty}\max_{[0,T]}\left(
\|n^k_t(t)-n_t(t)\|_{-\sigma}^2+\| n^k(t)-n(t)\|_{1-\sigma}^2 \right)=0,
\end{equation}
and
\begin{equation}\label{en-rel11}
\lim_{k\to\infty}\max_{[0,T]}\left(
\|B^k_t(t)-B_t(t)\|_{-\sigma}^2+\| B^k(t)-B(t)\|_{1-\sigma}^2 \right)=0,
\end{equation}
for every $\sigma>0$. In particular, this implies that
\begin{equation}\label{en-rel12}
(n^k_t(t);n^k(t);B_t^k(t); B(t)^k;E_t^k(t);E^k(t))\to (n_t(t);n(t);B_t(t); B(t);E_t(t);E(t))
\end{equation}
for each $t\in [0,T]$ weakly in $L_2\times H_1\times L_2\times H_1\times L_2\times H_2$.
From energy relation (\ref{en-rel}) and convergences (\ref{en-rel9})-(\ref{en-rel11}) it follows that 
\begin{eqnarray}
\label{en-rel13}
\lefteqn{\limsup\limits_{k\to \infty}\left |V(n_t, n, B_t, B, E_t)
+\int\limits_0^t f(n_t(\tau), B_t(\tau), E_t(\tau))d\tau-V(n^k_t, n^k, B^k_t, B^k, E^k_t)\right.}
\nonumber
\\
& & \qquad\qquad\qquad\qquad\qquad\qquad\qquad\qquad
\left.
-\int\limits_0^t f(n^k_t(\tau), B^k_t(\tau), E^k_t(\tau))d\tau\right |=0,
\end{eqnarray}
where 
\[
f(n_t, B_t, E_t)=2\gamma_1\|E_t(t)\|^2+\gamma_2\|n_t(t)\|^2
+\gamma_3\|B_t(t)\|^2.
\]
By the property of weak convergence from (\ref{en-rel12}) we get
\[
\int\limits_0^t f(n_t(\tau), B_t(\tau), E_t(\tau))d\tau\le 
\liminf\limits_{k\to\infty} \int\limits_0^t f(n^k_t(\tau), B^k_t(\tau), E^k_t(\tau))d\tau.
\]
Hence, from (\ref{en-rel13}) we obtain that
\begin{eqnarray*}
\lefteqn{
\liminf\limits_{k\to\infty} \left (V(n^k_t, n^k, B^k_t, B^k, E^k_t)
+\int\limits_0^t f(n_t(\tau), B_t(\tau), E_t(\tau))d\tau\right )
}
\\
& & \qquad
\le \lim\limits_{k\to\infty}\left (V(n^k_t, n^k, B^k_t, B^k, E^k_t)+\int\limits_0^t f(n^k_t(\tau), B^k_t(\tau), E^k_t(\tau))d\tau\right)
\\
& & \qquad
=V(n_t, n, B_t, B, E_t)
+\int\limits_0^t f(n_t(\tau), B_t(\tau), E_t(\tau))d\tau.
\end{eqnarray*}
Thus $\liminf\limits_{k\to\infty} V(n^k_t, n^k, B^k_t, B^k, E^k_t)
\le V(n_t, n, B_t, B, E_t)$.
However, the weak convergence (\ref{en-rel12}) implies  $\liminf\limits_{k\to\infty} V(n^k_t, n^k, B^k_t, B^k, E^k_t)
\ge V(n_t, n, B_t, B, E_t)$.
Therefore 
\[
\lim\limits_{k\to\infty} V(n^k_t, n^k, B^k_t, B^k, E^k_t)= V(n_t, n, B_t, B, E_t)
\]
for all $t\in [0, T]$.
Moreover from (\ref{en-rel13}) it follows that 
\begin{equation}\label{en-rel14}
\lim\limits_{k\to\infty} \left | \int\limits_0^t f(n_t(\tau), B_t(\tau), E_t(\tau))d\tau -\int\limits_0^t f(n^k_t(\tau), B^k_t(\tau), E^k_t(\tau))d\tau\right |=0
\end{equation}
for all $t\in [0, T]$.
Finally we note that Ascoli theorem, (\ref{en-rel8}) and (\ref{en-rel14}) imply (\ref{en-rel7}).

\qed

Now we are in position to prove (\ref{c-id}). Denote $(n_t, n, B_t, B, E)$ and $(n^k_t, n^k, B^k_t, B^k, E^k)$ by $U$ and $U^k$ respectively.
From (\ref{en-rel7}) it follows that $\|U^k\|_{C([0,T], \cH)}\to \|U\|_{C([0,T], \cH)}$  and from (\ref{en-rel12}) we conclude that $U^k\to U$ weakly in $\cH$ for each $t\in [0, T]$. 
It is easy to see that
\begin{eqnarray*}
\lefteqn{\max\limits_{[0,T]} \|U-U^k\|^2_{\cH}\le
\max\limits_{[0,T]}\left (\|U\|^2-\|U^k\|^2\right )
+2\max\limits_{[0,T]}\left |(U, U-U^k)\right |}
\\
& & \le \max\limits_{[0,T]}\left (\|U\|^2-\|U^k\|^2\right )
+C\max\limits_{[0,T]}\left \|P_N( U-U^k)\right \|
+C\max\limits_{[0,T]}\left \|Q_N( U-U^k)\right \|,
\end{eqnarray*}
where $P_N$ is the projector of the first $N$ basis vectors of $\cH$ and $Q_N=I-P_N$.
It implies that
\[
\limsup_{k\to\infty}\max\limits_{[0,T]} \|U-U^k\|^2_{\cH}\le
C\limsup_{k\to\infty}\max\limits_{[0,T]}\left \|Q_N( U-U^k)\right \|.
\]
Let $h_N(t)=\left\|Q_N(U(t)-U^k(t))\right\|$.
The sequence  $\{h_N(t)\}$ is a non-increasing  sequence of continuous functions
such that  $h_N(t)\to0$ for each $t\in [0,T]$ as $N\to\infty$.
Thus, by the Dini theorem
 $h_N(t)\to 0$ uniformly in $t$ and therefore  we obtain that
\[
\lim_{k\to\infty}\max_{[0,T]}|U(t)-U^k(t)|^2 =0
\]
which implies (\ref{c-id}).
This completes the proof of Theorem~\ref{t:continuos}.

%================================================================
%================================================================

%=================================================================

\end{document}